\documentclass[a4paper,reqno]{amsart}%
\usepackage{pgf,tikz}
\usepackage{amssymb}
\usepackage{amsmath}
\usepackage{amsfonts}
\usepackage{graphicx}
\usepackage{bm}%
\providecommand{\U}[1]{\protect\rule{.1in}{.1in}}
\usetikzlibrary{backgrounds}
\usetikzlibrary{arrows}
\let\oldmathbf\mathbf
\renewcommand{\mathbf}[1]{\boldsymbol{\oldmathbf{#1}}}
\newtheorem{theorem}{Theorem}

\newtheorem{corollary}[theorem]{Corollary}

\newtheorem{lemma}[theorem]{Lemma}

\newtheorem{proposition}[theorem]{Proposition}

\allowdisplaybreaks
\begin{document}
\title{On a sharp lemma of Cassels and Montgomery on manifolds}
\author[L. Brandolini]{Luca Brandolini}
\address{Dipartimento di Ingegneria Gestionale, dell'Informazione e della Produzione,
Universit\`a degli Studi di Bergamo, Viale Marconi 5, Dalmine BG, Italy}
\email{luca.brandolini@unibg.it}
\author[B. Gariboldi]{Bianca Gariboldi}
\address{Dipartimento di Ingegneria Gestionale, dell'Informazione e della Produzione,
Universit\`a degli Studi di Bergamo, Viale Marconi 5, Dalmine BG, Italy}
\email{biancamaria.gariboldi@unibg.it}
\author[G. Gigante]{Giacomo Gigante}
\address{Dipartimento di Ingegneria Gestionale, dell'Informazione e della Produzione,
Universit\`a degli Studi di Bergamo, Viale Marconi 5, Dalmine BG, Italy}
\email{giacomo.gigante@unibg.it}

\begin{abstract}
Let $\left(  \mathcal{M},g\right)  $ be a $d$-dimensional compact connected
Riemannian manifold and let $\left\{  \varphi_{m}\right\}  _{m=0}^{+\infty}$
be a complete sequence of orthonormal eigenfunctions of the Laplace-Beltrami
operator on $\mathcal{M}$. We show that there exists a positive constant $C$
such that for all integers $N$ and $X$ and for all finite sequences of $N$
points in $\mathcal{M}$, $\left\{  x\left(  j\right)  \right\}  _{j=1}^{N}$,
and positive weights $\left\{  a_{j}\right\}  _{j=1}^{N}$ we have%
\[
\sum_{m=0}^{X}\left\vert \sum_{j=1}^{N}a_{j}\varphi_{m}\left(  x\left(
j\right)  \right)  \right\vert ^{2}\geq\max\left\{  CX\sum_{j=1}^{N}a_{j}%
^{2},\left(  \sum_{j=1}^{N}a_{j}\right)  ^{2}\right\}  .
\]

\thanks{The authors have been supported by a GNAMPA 2019 project.}

\end{abstract}
\maketitle

\bigskip


Let $\left(  \mathcal{M},g\right)  $ be a $d$-dimensional compact connected
Riemannian manifold, with normalized Riemannian measure $\mu$ such that
$\mu(\mathcal{M})=1$, and Riemannian distance $d\left(  x,y\right)  $. Let
$\left\{  \lambda_{m}^{2}\right\}  _{m=0}^{+\infty}$ be the sequence of
eigenvalues of the (positive) Laplace-Beltrami operator $\Delta$, listed in
increasing order with repetitions, and let $\left\{  \varphi_{m}\right\}
_{m=0}^{+\infty}$ be an associated sequence of orthonormal eigenfunctions. In
particular $\varphi_{0}\equiv1$ and $\lambda_{0}=0$. This allows to define the
Fourier coefficients of $L^{1}(\mathcal{M})$ functions as
\[
\widehat{f}(\lambda_{m})=\int_{\mathcal{M}}f(x)\overline{\varphi_{m}(x)}%
d\mu(x)
\]
and the associated Fourier series
\[
\sum_{m=0}^{+\infty}\widehat{f}(\lambda_{m})\varphi_{m}(x).
\]

The main result of this paper is the following theorem.

\begin{theorem}
\label{main} There exists a positive constant $C$ such that for all integers
$N$ and $X$ and for all finite sequences of $N$ points in $\mathcal{M}$,
$\left\{  x\left(  j\right)  \right\}  _{j=1}^{N}$, and positive weights
$\left\{  a_{j}\right\}  _{j=1}^{N}$ we have%
\[
\sum_{m=0}^{X}\left\vert \sum_{j=1}^{N}a_{j}\varphi_{m}\left(  x\left(
j\right)  \right)  \right\vert ^{2}\geq\max\left\{  CX\sum_{j=1}^{N}a_{j}%
^{2},\left(  \sum_{j=1}^{N}a_{j}\right)  ^{2}\right\}  .
\]

\end{theorem}

Notice that the estimate
\[
\sum_{m=0}^{X}\left\vert \sum_{j=1}^{N}a_{j}\varphi_{m}\left(  x\left(
j\right)  \right)  \right\vert ^{2}\geq\left(  \sum_{j=1}^{N}a_{j}\right)
^{2}%
\]
is immediately obtained since for $m=0$ one has $\varphi_{0}\left(  x\right)
=1$ for all $x$ in $\mathcal{M}$. The essential part of the theorem is
therefore the estimate%
\begin{equation}
\sum_{m=0}^{X}\left\vert \sum_{j=1}^{N}a_{j}\varphi_{m}\left(  x\left(
j\right)  \right)  \right\vert ^{2}\geq CX\sum_{j=1}^{N}a_{j}^{2}.
\label{essential}%
\end{equation}

When $\mathcal{M}$ is the one-dimensional torus, the above theorem is
classical and goes back to the work of J. W. S. Cassels \cite{C}. This was
later extended to the higher dimensional torus by H. L. Montgomery, see e.g.
his book \cite{M}, or \cite{Giancarlo}. Recently, D. Bilyk, F. Dai, S.
Steinerberger \cite{BDS} extend the Cassels-Montgomery inequality to the case
of smooth compact $d$-dimensional Riemannian manifolds without boundary. More
precisely they show that there exists a positive constant $C$ such that for
all integers $N$ and $X$ and for all finite sequences of $N$ points in
$\mathcal{M}$, $\left\{  x\left(  j\right)  \right\}  _{j=1}^{N}$, and
positive weights $\{a_{j}\}_{j=1}^{N}$,%
\[
\sum_{m=0}^{X}\left\vert \sum_{j=1}^{N}a_{j}\varphi_{m}\left(  x\left(
j\right)  \right)  \right\vert ^{2}\geq C\frac{X}{(\log X)^{d/2}}\sum
_{j=1}^{N}a_{j}^{2}.
\]
This result should be compared with the following simple proposition.

\begin{proposition}
Let $X$ and $N$ be positive integers. For all positive weights $\{a_{j}%
\}_{j=1}^{N}$, there exists a sequence of points $\left\{  x\left(  j\right)
\right\}  _{j=1}^{N}$ in $\mathcal{M}$ such that%
\[
\sum_{m=0}^{X}\left\vert \sum_{j=1}^{N}a_{j}\varphi_{m}\left(  x\left(
j\right)  \right)  \right\vert ^{2}\leq X\sum_{j=1}^{N}a_{j}^{2}+\left(
\sum_{j=1}^{N}a_{j}\right)  ^{2}.
\]

\end{proposition}

\begin{proof}
Let%
\[
\Phi\left(  y_{1},\ldots,y_{N}\right)  =\sum_{m=1}^{X}\left\vert \sum
_{j=1}^{N}a_{j}\varphi_{m}\left(  y_{j}\right)  \right\vert ^{2}=\sum
_{m=1}^{X}\sum_{j,k=1}^{N}a_{j}a_{k}\varphi_{m}\left(  y_{j}\right)
\overline{\varphi_{m}\left(  y_{k}\right)  }.
\]
Since for $m\neq0$%
\[
\int_{\mathcal{M}}\varphi_{m}\left(  y_{j}\right)  dy_{j}=0,
\]
if $j\neq k$ we have%
\[
\int_{\mathcal{M}}\cdots\int_{\mathcal{M}}\varphi_{m}\left(  y_{j}\right)
\overline{\varphi_{m}\left(  y_{k}\right)  }dy_{1}\cdots dy_{N}=\int
_{\mathcal{M}}\int_{\mathcal{M}}\varphi_{m}\left(  y_{j}\right)
\overline{\varphi_{m}\left(  y_{k}\right)  }dy_{j}dy_{k}=0,
\]
while%
\[
\int_{\mathcal{M}}\cdots\int_{\mathcal{M}}\varphi_{m}\left(  y_{j}\right)
\overline{\varphi_{m}\left(  y_{j}\right)  }dy_{1}\cdots dy_{N}=\int
_{\mathcal{M}}\left\vert \varphi_{m}\left(  y_{j}\right)  \right\vert
^{2}dy_{j}=1.
\]
Hence,%
\begin{align*}
&  \int_{\mathcal{M}}\cdots\int_{\mathcal{M}}\Phi\left(  y_{1},\ldots
,y_{N}\right)  dy_{1}\cdots dy_{N}\\
&  =\sum_{m=1}^{X}\sum_{j=1}^{N}a_{j}^{2}\int_{\mathcal{M}}\left\vert
\varphi_{m}\left(  y_{j}\right)  \right\vert ^{2}dy_{j}=X\sum_{j=1}^{N}%
a_{j}^{2}.
\end{align*}
Therefore there exist points $\left\{  x\left(  j\right)  \right\}  _{j=1}%
^{N}$ such that%
\[
\Phi\left(  x\left(  1\right)  ,\ldots,x\left(  N\right)  \right)  \leq
X\sum_{j=1}^{N}a_{j}^{2}.
\]

\end{proof}

Our goal is therefore to remove the logarithmic loss in the above result of
Bilyk, Dai and Steinerberger, thus obtaining a sharp estimate.

The original proof by Montgomery in the case of the torus uses the Fej\'{e}r
kernel. A direct adaptation of this proof to the case of a general manifold
would require to construct a positive kernel of the form
\[
\sum_{m=0}^{X}c_{m}\varphi_{m}(x)\overline{\varphi_{m}(y)},
\]
but unfortunately this type of kernels is not available in a general manifold.
One could therefore withdraw, for example, the requirement that the spectrum
of the kernel be contained in the set $\{\lambda_{0}^{2},\ldots,\lambda
_{X}^{2}\}$. This is the strategy followed by Bilyk, Dai and Steinerberger
which use the heat kernel. Our strategy here is on the contrary to use a
kernel which is positive up to a negligible error, without dropping the
spectrum condition. The existence of such type of kernel can be proved by
means of the Hadamard parametrix for the wave operator on the manifold. In the
next section we introduce this construction.

\section{The Hadamard parametrix for the wave equation}

Following \cite[III, \S 17.4]{Hormander}, for $\nu=0,\,1,\,2,\ldots$, let us
call $E_{\nu}\left(  t,x\right)  $ the distribution defined as the inverse
Fourier-Laplace transform on $\mathbb{R}^{d+1}$ of $\nu!\left(  \left\vert
\xi\right\vert ^{2}-\tau^{2}\right)  ^{-\nu-1}$,
\[
E_{\nu}\left(  t,x\right)  =\nu!\left(  2\pi\right)  ^{-d-1}\int
_{\operatorname{Im}\tau=c<0}e^{i\left(  x\cdot\xi+t\tau\right)  }\left(
\left\vert \xi\right\vert ^{2}-\tau^{2}\right)  ^{-\nu-1}d\xi d\tau.
\]
Note that for $\nu=0$, this is exactly the fundamental solution of the wave
operator, see \cite[I, \S 6.2]{Hormander}. The next Proposition (see
\cite[III, Lemma 17.4.2]{Hormander}) gives more information about the
distributions $E_{\nu}$.

\begin{proposition}
\label{Lemma Hormander}~

\begin{itemize}
\item[(i)] $E_{\nu}$ is a homogeneous distribution of degree $2\nu-d+1$
supported in the forward light cone $\left\{  \left(  t,x\right)
\in\mathbb{R}^{1+d}:t>0,t^{2}\geq\left\vert x\right\vert ^{2}\right\}  $.

\item[(ii)] Moreover%
\[
E_{\nu}\left(  t,x\right)  =2^{-2\nu-1}\pi^{\left(  1-d\right)  /2}\chi
_{+}^{\nu+\left(  1-d\right)  /2}\left(  t^{2}-\left\vert x\right\vert
^{2}\right),  ~~~t>0,
\]
and $E_{\nu}\left(  t,x\right)  $ can be regarded as a smooth function of
$t>0$ with values in $\mathcal{D}^{\prime}\left(  \mathbb{R}^{d}\right)  $. In
particular if $\psi\in C_{0}^{\infty}\left(  \mathbb{R}^{1+d}\right)  $ then%
\[
\left\langle E_{\nu},\psi\right\rangle =2^{-2\nu-1}\pi^{\left(  1-d\right)
/2}\int_{0}^{+\infty}\left\langle \chi_{+}^{\nu+\left(  1-d\right)  /2}\left(
t^{2}-\left\vert \cdot\right\vert ^{2}\right)  ,\psi\left(  t,\cdot\right)
\right\rangle dt.
\]
Also%
\[
\partial_{t}^{k}E_{\nu}\left(  0^{+},\cdot\right)  =0~\text{for }k\leq2\nu
\]
and%
\[
\partial_{t}^{2\nu+1}E_{\nu}\left(  0^{+},\cdot\right)  =\nu!\delta_{0}.
\]

\item[(iii)] Finally, setting%
\[
\left\langle \check{E}_{\nu},\varphi\right\rangle :=\left\langle E_{\nu
},\check{\varphi}\right\rangle ,
\]
where
\[
\check{\varphi}\left(  t,x\right)  =\varphi\left(  -t,x\right)  ,
\]
the distributions $\left(  E_{\nu}-\check{E}_{\nu}\right)  \left(  t,x\right)
$ and $\partial_{t}\left(  E_{\nu}-\check{E}_{\nu}\right)  \left(  t,x\right)
$ can be regarded as continuous radial functions of $x$ with values in
$\mathcal{D}^{\prime}\left(  \mathbb{R}\right)  $. With a small abuse of
notation we will write $\left(  E_{\nu}-\check{E}_{\nu}\right)  \left(
\cdot,\left\vert x\right\vert \right)  $ and $\partial_{\cdot}\left(  E_{\nu
}-\check{E}_{\nu}\right)  \left(  \cdot,\left\vert x\right\vert \right)  $.
\end{itemize}
\end{proposition}

Let us clarify the meaning of the objects that appear in this proposition. Let
$\alpha\in\mathbb{C}$ be such that $\operatorname{Re}\alpha>-1$ and for every
test function $\varphi\in C_{0}^{\infty}\left(  \mathbb{R}\right)  $ define
the distribution $\chi_{+}^{\alpha}$ as%
\[
\left\langle \chi_{+}^{\alpha},\varphi\right\rangle =\frac{1}{\Gamma\left(
\alpha+1\right)  }\int_{0}^{+\infty}x^{\alpha}\varphi\left(  x\right)  dx.
\]
Integration by parts immediately gives
\[
\left\langle \chi_{+}^{\alpha},\varphi\right\rangle =-\left\langle \chi
_{+}^{\alpha+1},\varphi^{\prime}\right\rangle
\]
so that $\chi_{+}^{\alpha}$ can be extended to all $\alpha$ with
$\operatorname{Re}\alpha>-2$, and, repeating the argument, to the whole
complex plane (see \cite[I, \S 3.2]{Hormander} for the details).

Also, since the function $f(x,t)=t^{2}-|x|^{2}$ is a submersion of
$\mathbb{R}^{d+1}\setminus\{0\}$ in $\mathbb{R}$, then the pull-back $\chi
_{+}^{\alpha}(t^{2}-|x|^{2}):=f^{\ast}(\chi_{+}^{\alpha})\in\mathcal{D}%
^{\prime}(\mathbb{R}^{d+1}\setminus\{0\})$ is defined by the identity
\[
\left\langle f^{\ast}(\chi_{+}^{\alpha}),\varphi\right\rangle :=\left\langle
\chi_{+}^{\alpha},\int_{f^{-1}(\cdot)}\frac{\varphi(x,t)}{\Vert\nabla
f(x,t)\Vert}d\sigma(x,t)\right\rangle .
\]
We observe that by \cite[I, Theorem 3.23]{Hormander} the distribution
$\chi_{+}^{\nu+\left(  1-d\right)  /2}(t^{2}-|x|^{2})$ can be uniquely
extended to $\mathcal{D}^{\prime}(\mathbb{R}^{d+1})$ for $\nu=0,1,\ldots$.

Recall that distributions in $\mathcal{D}^{\prime}(\mathcal{M})$ can always be written as $u=\sum_{m=0}^{+\infty} c_m \varphi_m$, where the sequence $\{c_m\}$ is slowly increasing. Their action on smooth functions is given by
\[
\langle u,\phi\rangle=
 \sum_{m=0}^{+\infty}c_{m}\int\phi
{\varphi_{m}}.
\]
Consider the continuous linear map $\mathcal{K}_{t}:\mathcal{D}(\mathcal{M}%
)\rightarrow\mathcal{D}^{\prime}(\mathcal{M})$ defined by%
\[
\phi\mapsto\mathcal{K}_{t}\phi=\sum_{m}\cos\left(  \lambda_{m}t\right)
\widehat{\phi}\left(  \lambda_{m}\right)  \varphi_{m}.
\]
Observe that $\mathcal{K}_{t}\phi$ is in fact a smooth function and it is the
solution of the following Cauchy problem for the wave equation%
\[
\left\{
\begin{array}
[c]{l}%
\left(  \dfrac{\partial^{2}}{\partial t^{2}}+\Delta_{x}\right)  w\left(
t,x\right)  =0\\
w\left(  0,x\right)  =\phi\left(  x\right)  \\
\dfrac{\partial w}{\partial t}\left(  0,x\right)  =0.
\end{array}
\right.
\]
By the Schwartz kernel Theorem (see \cite[I, Theorem 5.2.1]{Hormander}), there
exists one and only one distribution $\cos(t\sqrt{\Delta})(x,y)\in
\mathcal{D}^{\prime}(\mathcal{M}\times\mathcal{M})$, called kernel of the map
$\mathcal{K}_{t}$, such that%
\begin{align*}
\langle\cos(t\sqrt{\Delta})(x,y),\eta(x)\phi(y)\rangle &  =\langle
\mathcal{K}_{t}\phi,\eta\rangle=\sum_{m}\cos\left(  \lambda_{m}t\right)
\widehat{\phi}\left(  \lambda_{m}\right)  \left\langle \varphi_{m}%
,\eta\right\rangle \\
&  =\sum_{m}\cos\left(  \lambda_{m}t\right)  \int\phi\left(  y\right)
\overline{\varphi_{m}\left(  y\right)  }dy\int\eta\left(  x\right)
\varphi_{m}\left(  x\right)  dx.
\end{align*}
This immediately implies that%
\[
\cos(t\sqrt{\Delta})(x,y)=\sum_{m}\cos(\lambda_{m}t)\varphi_{m}%
(x)\overline{\varphi_{m}(y)},
\]
and the identity is of course in the sense of distributions in $\mathcal{D}%
^{\prime}(\mathcal{M}\times\mathcal{M})$.

Hadamard's construction of the parametrix for the wave operator allows to
describe for small values of time $t$ the singularities of $\cos\left(
t\sqrt{\Delta}\right)  \left(  x,y\right)  $.

\begin{theorem}
[{see {\cite[Theorem 3.1.5]{sogge}}}]\label{sogge}Given a $d$-dimensional
Riemannian manifold $\left(  \mathcal{M},g\right)  $, there exists
$\varepsilon>0$ and functions $\alpha_{\nu}\in\mathcal{C}^{\infty}%
(\mathcal{M}\times\mathcal{M})$, so that if $Q>d+3$ the following holds. Let%
\[
K_{Q}\left(  t,x,y\right)  =\sum_{\nu=0}^{Q}\alpha_{\nu}\left(  x,y\right)
\partial_{t}\left(  E_{\nu}-\check{E}_{\nu}\right)  \left(  t,d\left(
x,y\right)  \right)
\]
and%
\[
R_{Q}\left(  t,x,y\right)  =\cos\left(  t\sqrt{\Delta}\right)  \left(
x,y\right)  -K_{Q}\left(  t,x,y\right)  ,
\]
then $R_{Q}\in\mathcal{C}^{Q-d-3}\left(  \left[  -\varepsilon,\varepsilon
\right]  \times\mathcal{M}\times\mathcal{M}\right)  $ and%
\[
\left\vert \partial_{t,x,y}^{\beta}R_{Q}\left(  t,x,y\right)  \right\vert \leq
C\left\vert t\right\vert ^{2Q+2-d-\left\vert \beta\right\vert }.
\]
Furthermore $\alpha_{0}\left(  x,y\right)  >0$ in $\mathcal{M}\times
\mathcal{M}$.
\end{theorem}

Observe that $K_{Q}\left(  t,x,y\right) $, by Proposition \ref{Lemma Hormander} (iii), defines a distribution on
$\mathbb{R\times}\mathcal{M}\times\mathcal{M}$ via the identity%
\[
\left\langle K_{Q},\varphi\right\rangle =\iint_{\mathcal{M}\times\mathcal{M}%
}\left\langle K_{Q}\left(  \cdot,x,y\right)  ,\varphi\left(  \cdot,x,y\right)
\right\rangle dxdy.
\]
However this distribution describes the singularities of the kernel
$\cos\left(  t\sqrt{\Delta}\right)  \left(  x,y\right)  $ only for small time.

\section{Notations and Fourier transforms}

Let us introduce some notation. If $f$ and $g$ are integrable functions on
$\mathbb{R}^{d}$, we shall denote their convolution by%
\[
f\ast_{d}g(x)=\int_{\mathbb{R}^{d}}f(x-y)g(y)dy.
\]
We define the cosine transform of smooth even functions on $\mathbb{R}$ as%
\[
\mathcal{C}f\left(  t\right)  =\int_{0}^{\infty}f\left(  s\right)  \cos\left(
st\right)  ds
\]
with inverse%
\[
\mathcal{C}^{-1}f\left(  s\right)  =\frac{2}{\pi}\int_{0}^{\infty}f\left(
t\right)  \cos\left(  st\right)  dt.
\]
For smooth functions on $\mathbb{R}^{d}$ we will use a slightly different
normalization, and we define the Fourier transform and its inverse as
\begin{align*}
\mathcal{F}_{d}f\left(  \xi\right)   &  =\int_{\mathbb{R}^{d}}f\left(
x\right)  e^{-2\pi ix\cdot\xi}dx,\\
\mathcal{F}_{d}^{-1}f\left(  x\right)   &  =\int_{\mathbb{R}^{d}}f\left(
\xi\right)  e^{2\pi ix\cdot\xi}d\xi.
\end{align*}
For radial functions $f\left(  x\right)  =f_{0}\left(  \left\vert x\right\vert
\right)  $, the above Fourier transform reduces essentially to the Hankel
transform, given by (see. \cite[Chapter 4, Theorem 3.3]{SW})
\begin{align}
\mathcal{F}_{d}f\left(  \xi\right)   &  =2\pi\left\vert \xi\right\vert
^{-\frac{d-2}{2}}\int_{0}^{\infty}f_{0}\left(  s\right)  J_{\frac{d-2}{2}%
}\left(  2\pi\left\vert \xi\right\vert s\right)  s^{\frac{d}{2}}%
ds,\label{Hankel}\\
\mathcal{F}_{d}^{-1}f\left(  x\right)   &  =2\pi\left\vert x\right\vert
^{-\frac{d-2}{2}}\int_{0}^{\infty}f_{0}\left(  s\right)  J_{\frac{d-2}{2}%
}\left(  2\pi\left\vert x\right\vert s\right)  s^{\frac{d}{2}}ds.\nonumber
\end{align}
In the future, with an abuse of notation, we will identify the function $f$
with its radial profile $f_{0}$ and write $\mathcal{F}_{d}f\left(  \left\vert
\xi\right\vert \right)  $ instead of $\mathcal{F}_{d}f\left(  \xi\right)  .$
One can easily show that
\[
\mathcal{F}_{1}f\left(  t\right)  =2\mathcal{C}f\left(  2\pi t\right)  .
\]

In the proof of Theorem \ref{main} we need the inverse cosine transform of the
distribution $\partial_{t}\left(  E_{\nu}-\check{E}_{\nu}\right)  $. By
Proposition \ref{Lemma Hormander} (iii),  $\partial_{t}\left(  E_{\nu}%
-\check{E}_{\nu}\right)  \left(  t,z\right)  $ can be seen as a continuous
function of $z$ into $\mathcal{D}^{\prime}\left(  \mathbb{R}\right)  .$ In the
following Lemma we compute for every fixed $z$ the inverse cosine transform of
this distribution.

\begin{lemma}
\label{Trasf E_nu}Let $0\leq\nu<d/2$. For every $z\in\mathbb{R}^{d}$,
$\mathcal{C}^{-1}\left(  \partial_{\cdot}\left(  E_{\nu}-\check{E}_{\nu
}\right)  \left(  \cdot,z\right)  \right)  $ is a function and for all
$t\in\mathbb{R}$%
\begin{equation}
\mathcal{C}^{-1}\left(  \partial_{\cdot}\left(  E_{\nu}-\check{E}_{\nu
}\right)  \left(  \cdot,z\right)  \right)  \left(  t\right)  =\pi
^{-d/2}2^{-\nu-d/2}\left\vert t\right\vert ^{-2\nu-1+d}\frac{J_{-\nu
+d/2-1}\left(  t\left\vert z\right\vert \right)  }{\left(  t\left\vert
z\right\vert \right)  ^{-\nu+d/2-1}}. \label{identity}%
\end{equation}

\end{lemma}

\begin{proof}
Since by Proposition \ref{Lemma Hormander} (i) and (iii)%
\[
E_{\nu}\left(  t,z\right)  =2^{-2\nu-1}\pi^{\left(  1-d\right)  /2}\chi
_{+}^{\nu+\left(  1-d\right)  /2}\left(  t^{2}-\left\vert z\right\vert
^{2}\right)  \quad t>0,
\]
if $\left(  d-1\right)  /2<\nu<d/2$, then
\[
\partial_{t}\left(  E_{\nu}-\check{E}_{\nu}\right)  \left(  t,z\right)
\]
is an even, locally integrable function in $t$, vanishing at $\infty$, so that
its cosine transform is (see \cite[Formula 11, Table 1.3, Chapter 1, page
12]{bateman}),
\begin{align*}
&  \mathcal{C}^{-1}\left(  \partial_{\cdot}\left(  E_{\nu}-\check{E}_{\nu
}\right)  \left(  \cdot,z\right)  \right)  \left(  s\right) \\
&  =\frac{2}{\pi}\int_{0}^{+\infty}\partial_{t}E_{\nu}\left(  t,z\right)
\cos\left(  st\right)  dt\\
&  =\frac{2^{-2\nu+1}}{\pi^{\left(  d+1\right)  /2}}\frac{1}{\Gamma\left(
\nu+\left(  1-d\right)  /2\right)  }\int_{\left\vert z\right\vert }^{+\infty
}t\left(  t^{2}-\left\vert z\right\vert ^{2}\right)  ^{\nu+\left(  1-d\right)
/2-1}\cos\left(  st\right)  dt\\
&  =\pi^{-d/2}2^{-\nu-d/2}\left\vert s\right\vert ^{-2\nu-1+d}\frac
{J_{-\nu+d/2-1}\left(  s\left\vert z\right\vert \right)  }{\left(  s\left\vert
z\right\vert \right)  ^{-\nu+d/2-1}}.
\end{align*}
Observe now that the distribution $\chi_{+}^{\nu+\left(  1-d\right)  /2}$ in
$\mathcal{D}^{\prime}\left(  \mathbb{R}\right)  $ is entire in the variable
$\nu$, and so is the distribution $\partial_{t}\left(  E_{\nu}-\check{E}_{\nu}\right)  \left(  t,z\right)$ in $\mathcal{D}^{\prime}\left(
\mathbb{R}\right)  $
for fixed $z$. This implies that also the cosine transform%
\[
\mathcal{C}^{-1}\left(  \partial_{\cdot}\left(  E_{\nu}-\check{E}_{\nu
}\right)  \left(  \cdot,z\right)  \right)  \left(  s\right)
\]
can be analytically extended to all complex values of $\nu$ (see \cite[Note 1,
page 171]{gelfand-shilov}). This analytic extension coincides therefore with
the analytic extension of the distribution%
\[
\pi^{-d/2}2^{-\nu-d/2}\left\vert s\right\vert ^{-2\nu-1+d}\frac{J_{-\nu
+d/2-1}\left(  s\left\vert z\right\vert \right)  }{\left(  s\left\vert
z\right\vert \right)  ^{-\nu+d/2-1}}.
\]
Observe that this is the product of the locally integrable function
$\left\vert s\right\vert ^{-2\nu-1+d}$ (recall that $\nu<d/2$) with the smooth
function $\pi^{-d/2}2^{-\nu-d/2}\frac{J_{-\nu+d/2-1}\left(  s\left\vert
z\right\vert \right)  }{\left(  s\left\vert z\right\vert \right)
^{-\nu+d/2-1}}$, which is analytic in $\nu\in\mathbb{C}$.

Thus, the identity
\[
\mathcal{C}^{-1}\left(  \partial_{\cdot}\left(  E_{\nu}-\check{E}_{\nu
}\right)  \left(  \cdot,z\right)  \right)  \left(  s\right)  =\pi
^{-d/2}2^{-\nu-d/2}\left\vert s\right\vert ^{-2\nu-1+d}\frac{J_{-\nu
+d/2-1}\left(  s\left\vert z\right\vert \right)  }{\left(  s\left\vert
z\right\vert \right)  ^{-\nu+d/2-1}}%
\]
holds for all $\nu<d/2$.
\end{proof}

\section{Proof of the main result}

It suffices to show the main inequality (\ref{essential}) for any positive
integer $N$ and for any integer $X$ sufficiently large. Indeed, if $1\leq
X<X_{0}$
\begin{align*}
\sum_{m=0}^{X}\left\vert \sum_{j=1}^{N}a_{j}\varphi_{m}\left(  x\left(
j\right)  \right)  \right\vert ^{2}  &  \geq\left\vert \sum_{j=1}^{N}%
a_{j}\varphi_{0}\left(  x\left(  j\right)  \right)  \right\vert ^{2}=\left(
\sum_{j=1}^{N}a_{j}\right)  ^{2}\\
&  \geq\sum_{j=1}^{N}a_{j}^{2}\geq\frac{1}{X_{0}}X\sum_{j=1}^{N}a_{j}^{2}.
\end{align*}

Let $\kappa$ be a positive integer that we will choose later and let $Y=\kappa
X$. By \cite[Theorem 2]{GL}, we can split the manifold $\mathcal{M}$ into $Y$
disjoint regions $\left\{  \mathcal{R}_{i}\right\}  _{i=1}^{Y}$ with measure
$\left\vert \mathcal{R}_{i}\right\vert =1/Y$ and such that each region
contains a ball of radius $c_{1}Y^{-1/d}$ and is contained in a ball of radius
$c_{2}Y^{-1/d}$, for appropriate values of $c_{1}$ and $c_{2}$ independent of
$Y$. Let us call $\left\{  \mathcal{B}_{r}\right\}  _{r=1}^{R}$ the sequence
of all the regions of the above collection $\left\{  \mathcal{R}_{i}\right\}
_{i=1}^{Y}$ which contain at least one of the points $x\left(  j\right)  $. We
call $K_{r}$ the cardinality of the set $\left\{  j=1,\ldots,N:x\left(
j\right)  \in\mathcal{B}_{r}\right\}  $ and $S_{r}$ the sum of the weights
$\{a_{j}\}$ corresponding to points $x\left(  j\right)  \in\mathcal{B}_{r}$.
Without loss of generality we can assume that
\[
S_{1}\geq S_{2}\geq\ldots\geq S_{R}>0.
\]
We rename the sequence $\left\{  x\left(  j\right)  \right\}  _{j=1}^{N}$ as
\[
\left\{  x_{r,j}\right\}  _{\substack{r=1,\ldots,R\\j=1,\ldots,K_{r}}}
\]
with $x_{r,j}\in\mathcal{B}_{r}$ for all $j=1,\ldots,K_{r}$, and the sequence
$\left\{  a_{j}\right\}  _{j=1}^{N}$ as%
\[
\left\{  a_{r,j}\right\}  _{\substack{r=1,\ldots,R\\j=1,\ldots,K_{r}}}.
\]
Observe that
\[
S_{r}=\sum_{j=1}^{K_{r}}a_{r,j}.
\]
Inequality (\ref{essential}) is an immediate consequence of the following%
\begin{equation}
\sum_{m=0}^{X}\left\vert \sum_{r=1}^{R}\sum_{j=1}^{K_{r}}a_{r,j}\varphi
_{m}\left(  x_{r,j}\right)  \right\vert ^{2}\geq CX\sum_{r=1}^{R}\left(\sum
_{j=1}^{K_{r}}a_{r,j}\right)^{2}.\label{lemma}%
\end{equation}
Let $\psi$ be a smooth radial function on $\mathbb{R}^{d}$ compactly supported
in the ball $B\left(  0,1/2\right)  =\left\{  x\in\mathbb{R}^{d}:\left\vert
x\right\vert \leqslant1/2\right\}  $ such that $\left\Vert \psi\right\Vert
_{2}=1$ and $\int\psi>0$, and set $H\left(  x\right)  =\psi\ast_{d}\psi\left(
x\right)  $. Then clearly $H$ is radial, compactly supported in $B\left(
0,1\right)  $, $H\left(  x\right)  \leq1$ for all $x\in\mathcal{M}$, and $H\left(
0\right)  =1$. Moreover its Fourier transform is $\mathcal{F}_{d}H\left(
\xi\right)  =\left(  \mathcal{F}_{d}\psi\left(  \xi\right)  \right)  ^{2}%
\geq0$ for all $\xi\in\mathbb{R}^{d}$, and has fast decay at infinity with all
its derivatives.

If we now identify $H\left(  x\right)  $ with its radial profile, we can write%
\begin{align}
&  \sum_{m=0}^{X}\left\vert \sum_{r=1}^{R}\sum_{j=1}^{K_{r}}a_{r,j}\varphi
_{m}\left(  x_{r,j}\right)  \right\vert ^{2}\nonumber\\
&  \geq\sum_{m=0}^{+\infty}H\left(  \frac{\lambda_{m}}{\lambda_{X}}\right)
\left\vert \sum_{r=1}^{R}\sum_{j=1}^{K_{r}}a_{r,j}\varphi_{m}\left(
x_{r,j}\right)  \right\vert ^{2}\label{somma}\\
&  =\sum_{r=1}^{R}\sum_{j=1}^{K_{r}}\sum_{s=1}^{R}\sum_{i=1}^{K_{s}}%
a_{r,j}a_{s,i}\left(  \sum_{m=0}^{+\infty}H\left(  \frac{\lambda_{m}}%
{\lambda_{X}}\right)  \varphi_{m}\left(  x_{r,j}\right)  \overline{\varphi
_{m}\left(  x_{s,i}\right)  }\right)  .\nonumber
\end{align}
Let us define the kernel%
\begin{equation}
F_{X}\left(  x,y\right)  :=\sum_{m=0}^{+\infty}H\left(  \frac{\lambda_{m}%
}{\lambda_{X}}\right)  \varphi_{m}\left(  x\right)  \overline{\varphi
_{m}\left(  y\right)  }.\label{Fx}%
\end{equation}
We will estimate $F_{X}\left(  x,y\right)  $ using the parametrix for the wave
operator described in the previous section. For this, one would need that the
Fourier cosine transform of $H\left(  \frac{\cdot}{\lambda_{X}}\right)  $ have
small support. This of course cannot be achieved, having $H$ itself compact
support. For this reason we pick $\eta=\mathcal{F}_{d}\phi$ where $\phi\left(
\xi\right)  $ is a nonnegative smooth radial function supported in $B\left(
0,\varepsilon/2\pi\right)  $ and such that $\phi\left(  \xi\right)  =1$ in
$B\left(  0,\varepsilon/4\pi\right)  $ and define%
\[
\widetilde{H}\left(  x\right)  =H\left(  \frac{\cdot}{\lambda_{X}}\right)
\ast_{d}\eta\left(  x\right)  .
\]
The reason for taking a $d$-dimensional convolution will be clarified in Lemma
\ref{Lemma F1} where we use the fact that $\mathcal{F}_{d}\widetilde{H}\geq0$.

Observe that $\operatorname*{supp}\mathcal{F}_{d}\widetilde{H}\subset B\left(
0,\varepsilon/2\pi\right)  $. It is remarkable that the cosine transform of
$\widetilde{H}$ has support in $\left[  0,\varepsilon\right]  $ and is nonnegative.

\begin{lemma}
\label{support} $\mathcal{C}^{-1}\widetilde{H}\left(  \rho\right)  \geq0$ for
$\rho\geq0$ and $\mathcal{C}^{-1}\widetilde{H}\left(  \rho\right)  =0$ for
$\rho>\varepsilon$.
\end{lemma}

\begin{proof}
It is known (see \cite[eq. (3.9)]{st} ) that for $d>d^{\prime}\geq1,$%
\begin{equation}
\mathcal{F}_{d^{\prime}}\left(  \mathcal{F}_{d}g\right)  \left(  s\right)
=c_{d,d^{\prime}}\int_{s}^{+\infty}\left(  r^{2}-s^{2}\right)  ^{\left(
d-d^{\prime}\right)  /2-1}rg\left(  r\right)  dr. \label{Transplantation}%
\end{equation}
Let now $g\left(  r\right)  =\mathcal{F}_{d}\widetilde{H}\left(  r\right)  $.
Since $\widetilde{H}\left(  s\right)  =\mathcal{F}_{d}g\left(  s\right)  $ and
the cosine transform is essentially $\mathcal{F}_{1}$ we obtain%
\[
\mathcal{C}^{-1}\widetilde{H}\left(  \rho\right)  =\frac{1}{\pi}%
\mathcal{F}_{1}\mathcal{F}_{d}g\left(  \frac{\rho}{2\pi}\right)
\]
and the thesis follows immediately from (\ref{Transplantation}), the fact that
$g\left(  r\right)  \geq0$ and the fact that $g\left(  r\right)  =0$ for
$r>\varepsilon/2\pi$.
\end{proof}

Let us go back to the kernel $F_{X}$,
\begin{align*}
F_{X}\left(  x,y\right)   &  =\sum_{m=0}^{+\infty}H\left(  \frac{\lambda_{m}%
}{{\lambda_{X}}}\right)  \varphi_{m}\left(  x\right)  \overline{\varphi
_{m}\left(  y\right)  }\\
&  =\sum_{m=0}^{+\infty}\widetilde{H}\left(  \lambda_{m}\right)  \varphi
_{m}\left(  x\right)  \overline{\varphi_{m}\left(  y\right)  }+\sum
_{m=0}^{+\infty}\left(  H\left(  \frac{\lambda_{m}}{{\lambda_{X}}}\right)
-\widetilde{H}\left(  \lambda_{m}\right)  \right)  \varphi_{m}\left(
x\right)  \overline{\varphi_{m}\left(  y\right)  }%
\end{align*}
Since%
\[
\widetilde{H}\left(  \lambda\right)  =\int_{0}^{+\infty}\mathcal{C}%
^{-1}\widetilde{H}\left(  t\right)  \cos\left(  \lambda t\right)  dt,
\]
and $\mathcal{C}^{-1}\widetilde{H}\left(  t\right)  $ is supported in $\left[
-\varepsilon,\varepsilon\right]  $, by Theorem \ref{sogge}, we can write
\begin{align*}
&  \sum_{m=0}^{+\infty}\widetilde{H}\left(  \lambda_{m}\right)  \varphi
_{m}\left(  x\right)  \overline{\varphi_{m}\left(  y\right)  }\\
&  =\sum_{m=0}^{+\infty}\int_{0}^{+\infty}\mathcal{C}^{-1}\widetilde{H}\left(
t\right)  \cos\left(  \lambda_{m}t\right)  dt\varphi_{m}\left(  x\right)
\overline{\varphi_{m}\left(  y\right)  }\\
&  =\frac{1}{2}\sum_{m=0}^{+\infty}\left\langle \cos\left(  \lambda_{m}%
\cdot\right)  \varphi_{m}\left(  x\right)  \overline{\varphi_{m}\left(
y\right)  },\mathcal{C}^{-1}\widetilde{H}\right\rangle \\
&  =\frac{1}{2}\left\langle \cos\left(  \cdot\sqrt{\Delta}\right)  \left(
x,y\right)  ,\mathcal{C}^{-1}\widetilde{H}\right\rangle \\
&  =\frac{1}{2}\sum_{\nu=0}^{Q}\alpha_{\nu}\left(  x,y\right)  \Omega_{\nu
}\left(  x,y\right)  +\int_{0}^{\varepsilon}R_{Q}\left(  t,x,y\right)
\mathcal{C}^{-1}\widetilde{H}\left(  t\right)  dt
\end{align*}
where we set%
\[
\Omega_{\nu}\left(  x,y\right)  =\left\langle \partial_{t}\left(  E_{\nu
}-\check{E}_{\nu}\right)  \left(  \cdot,d\left(  x,y\right)  \right)
,\mathcal{C}^{-1}\widetilde{H}\right\rangle .
\]
We can therefore decompose the kernel $F_{X}$ as follows%
\[
F_{X}\left(  x,y\right)  =\sum_{n=1}^{5}F_{n}\left(  x,y\right)  ,
\]
where%
\begin{align*}
F_{1}\left(  x,y\right)   &  =\frac{1}{2}\alpha_{0}\left(  x,y\right)
\Omega_{0}\left(  x,y\right) , \\
F_{2}\left(  x,y\right)   &  =\frac{1}{2}\sum_{1\leq\nu<d/2}\alpha_{\nu
}\left(  x,y\right)  \Omega_{\nu}\left(  x,y\right),  \\
F_{3}\left(  x,y\right)   &  =\frac{1}{2}\sum_{d/2\leq\nu\leq Q}\alpha_{\nu
}\left(  x,y\right)  \Omega_{\nu}\left(  x,y\right),  \\
F_{4}\left(  x,y\right)   &  =\int_{0}^{\varepsilon}R_{Q}\left(  t,x,y\right)
\mathcal{C}^{-1}\widetilde{H}\left(  t\right)  dt,\\
F_{5}\left(  x,y\right)   &  =\sum_{m=0}^{+\infty}\left(  H\left(
\frac{\lambda_{m}}{{\lambda_{X}}}\right)  -\widetilde{H}\left(  \lambda
_{m}\right)  \right)  \varphi_{m}\left(  x\right)  \overline{\varphi
_{m}\left(  y\right)  .}%
\end{align*}
Recalling (\ref{somma}) and (\ref{Fx}) we have%
\begin{align*}
  \sum_{m=0}^{X}\left\vert \sum_{r=1}^{R}\sum_{j=1}^{K_{r}}a_{r,j}\varphi
_{m}\left(  x_{r,j}\right)  \right\vert ^{2}&=\sum_{r=1}^{R}\sum_{j=1}^{K_{r}%
}\sum_{s=1}^{R}\sum_{i=1}^{K_{s}}a_{r,j}a_{s,i}F_{X}\left(  x_{r,j}%
,x_{s,i}\right)  \\
&  =\sum_{n=1}^{5}\sum_{r=1}^{R}\sum_{j=1}^{K_{r}}\sum_{s=1}^{R}\sum
_{i=1}^{K_{s}}a_{r,j}a_{s,i}F_{n}\left(  x_{r,j},x_{s,i}\right)  .
\end{align*}
We start estimating the term with $F_{1}$ which is the positive part of the
kernel and gives the main contribution.

\begin{lemma}
\label{Lemma F1}For $\kappa$ large enough there exist $X_{0}>0$ and $C>0$ such
that for every $X>X_{0}$%
\[
\sum_{r=1}^{R}\sum_{j=1}^{K_{r}}\sum_{s=1}^{R}\sum_{i=1}^{K_{s}}a_{r,j}%
a_{s,i}F_{1}\left(  x_{r,j},x_{s,i}\right)  \geq CX\sum_{r=1}^{R}\left(
\sum_{j=1}^{K_{r}}a_{r,j}\right)  ^{2}.
\]

\end{lemma}

\begin{proof}
First of all we show that $\Omega_{0}\left(  x,y\right)  $ is positive.
Indeed, by Lemma \ref{Trasf E_nu} and (\ref{Hankel}), for every $x,y\in
\mathcal{M},$
\begin{align*}
\Omega_{0}\left(  x,y\right)   &  =\left\langle \mathcal{C}^{-1}\left(
\partial_{t}\left(  E_{0}-\check{E}_{0}\right)  \left(  \cdot,d\left(
x,y\right)  \right)  \right)  ,\widetilde{H}\right\rangle \\
&  =2\int_{0}^{+\infty}\mathcal{C}^{-1}\left(  \partial_{t}\left(
E_{0}-\check{E}_{0}\right)  \left(  \cdot,d\left(  x,y\right)  \right)
\right)  \left(  t\right)  \widetilde{H}\left(  t\right)  dt\\
&  =\frac{2}{d\left(  x,y\right)  ^{d/2-1}\left(  2\pi\right)  ^{d/2}}\int
_{0}^{+\infty}J_{d/2-1}\left(  2\pi\frac{d\left(  x,y\right)  }{2\pi}t\right)
\widetilde{H}\left(  t\right)  t^{d/2}dt\\
&  =\frac{2}{\left(  2\pi\right)  ^{d}}\mathcal{F}_{d}\widetilde{H}\left(
\frac{d\left(  x,y\right)  }{2\pi}\right)  \\
&  =\frac{2}{\left(  2\pi\right)  ^{d}}\lambda_{X}^{d}\mathcal{F}_{d}H\left(
\frac{\lambda_{X}d\left(  x,y\right)  }{2\pi}\right)  \mathcal{F}_{d}%
\eta\left(  \frac{d\left(  x,y\right)  }{2\pi}\right)  \geq0.
\end{align*}
Since also $\alpha_{0}\left(  x,y\right)  $ is positive, we can disregard
off-diagonal terms,%
\begin{align*}
&  \sum_{r=1}^{R}\sum_{j=1}^{K_{r}}\sum_{s=1}^{R}\sum_{i=1}^{K_{s}}%
a_{r,j}a_{s,i}F_{1}\left(  x_{r,j},x_{s,i}\right)\\ & =\frac{1}{2}\sum_{r=1}%
^{R}\sum_{j=1}^{K_{r}}\sum_{s=1}^{R}\sum_{i=1}^{K_{s}}a_{r,j}a_{s,i}\alpha
_{0}\left(  x_{r,j},x_{s,i}\right)  \Omega_{0}\left(  x_{r,j},x_{s,i}\right)
\\
&  \geqslant\frac{1}{2}\sum_{r=1}^{R}\sum_{j=1}^{K_{r}}\sum_{i=1}^{K_{r}%
}a_{r,j}a_{r,i}\alpha_{0}\left(  x_{r,j},x_{r,i}\right)  \Omega_{0}\left(
x_{r,j},x_{r,i}\right)  \\
&  =\frac{\lambda_{X}^{d}}{\left(  2\pi\right)  ^{d}}\sum_{r=1}^{R}\sum
_{j=1}^{K_{r}}\sum_{i=1}^{K_{r}}a_{r,j}a_{r,i}\alpha_{0}\left(  x_{r,j}%
,x_{r,i}\right)  \mathcal{F}_{d}H\left(  \frac{\lambda_{X}d\left(
x_{r,j},x_{r,i}\right)  }{2\pi}\right)  \mathcal{F}_{d}\eta\left(
\frac{d\left(  x_{r,j},x_{r,i}\right)  }{2\pi}\right)  .
\end{align*}
By Weyl's estimate (see e.g. \cite[III, Corollary 17.5.8]{Hormander})
$\lambda_{X}\sim X^{1/d}$. Thus, if $x,y\in\mathcal{B}_{r}$ then%
\[
\lambda_{X}d(x,y)\leq\lambda_{X}2c_{2}\left(  \kappa X\right)  ^{-1/d}\leq
c_{3}\kappa^{-1/d}.
\]
Let $\kappa$ large enough so that if $x,y\in\mathcal{B}_{r}$%
\[
\mathcal{F}_{d}H\left(  \frac{\lambda_{X}d\left(  x,y\right)  }{2\pi}\right)
=\left(  \mathcal{F}_{d}\psi\right)  ^{2}\left(  \frac{\lambda_{X}d(x,y)}%
{2\pi}\right)  \geq\frac{\left(  \mathcal{F}_{d}\psi\right)  ^{2}\left(
0\right)  }{2}>0
\]
and%
\[
\frac{d\left(  x,y\right)  }{2\pi}\leq\frac{\varepsilon}{4\pi},%
\]
so that
\[
\mathcal{F}_{d}\eta\left(  \frac{d\left(  x,y\right)  }{2\pi}\right)  =1.
\]
Finally,%
\begin{align*}
&  \sum_{r=1}^{R}\sum_{j=1}^{K_{r}}\sum_{s=1}^{R}\sum_{i=1}^{K_{s}}%
a_{r,j}a_{s,i}F_{1}\left(  x_{r,j},x_{s,i}\right)  \\
&  \geq CX\sum_{r=1}^{R}\sum_{j=1}^{K_{r}}\sum_{i=1}^{K_{r}}a_{r,j}%
a_{r,i}=CX\sum_{r=1}^{R}\left(  \sum_{j=1}^{K_{r}}a_{r,j}\right)  ^{2}.
\end{align*}

\end{proof}

The following lemmas show that the contributions given by the terms with $F_2$, $F_3$, $F_4$, $F_5$
are negligible.

\begin{lemma}
There exist $C>0$ and $X_{0}>0$ such that for every $X>X_{0}$%
\[
\left\vert \sum_{r=1}^{R}\sum_{j=1}^{K_{r}}\sum_{s=1}^{R}\sum_{i=1}^{K_{s}%
}a_{r,j}a_{s,i}F_{2}\left(  x_{r,j},x_{s,i}\right)  \right\vert \leq
CX^{1-2/d}\sum_{r=1}^{R}\left(  \sum_{j=1}^{K_{r}}a_{r,j}\right)  ^{2}.
\]

\end{lemma}

\begin{proof}
We will show that for every integer $\nu$, $1\leq\nu<d/2,$
\[
\left\vert \sum_{r=1}^{R}\sum_{j=1}^{K_{r}}\sum_{s=1}^{R}\sum_{i=1}^{K_{s}%
}a_{r,j}a_{s,i}\alpha_{\nu}\left(  x_{r,j},x_{s,i}\right)  \Omega_{\nu}\left(
x_{r,j},x_{s,i}\right)  \right\vert \leq CX^{1-2\nu/d}\sum_{r=1}^{R}\left(
\sum_{j=1}^{K_{r}}a_{r,j}\right)  ^{2}.
\]
By Lemma \ref{Trasf E_nu}, for every $x,y\in\mathcal{M}$,
\begin{align*}
\Omega_{\nu}\left(  x,y\right)   &  =\left\langle \mathcal{C}^{-1}\left(
\partial_{\cdot}\left(  E_{\nu}-\check{E}_{\nu}\right)  \left(  \cdot,d\left(
x,y\right)  \right)  \right)  ,\widetilde{H}\right\rangle \\
&  =2\int_{0}^{+\infty}\pi^{-d/2}2^{-\nu-d/2}\left\vert t\right\vert
^{-2\nu-1+d}\frac{J_{-\nu+d/2-1}\left(  d\left(  x,y\right)  t\right)
}{\left(  d\left(  x,y\right)  t\right)  ^{-\nu+d/2-1}}\widetilde{H}\left(
t\right)  dt\\
&  =\pi^{\nu-d}2^{1-d}\mathcal{F}_{d-2\nu}\widetilde{H}\left(  \frac{d\left(
x,y\right)  }{2\pi}\right)  \\
&  =\pi^{\nu-d}2^{1-d}\mathcal{F}_{d-2\nu}\left(  H\left(  \frac{\cdot
}{\lambda_{X}}\right)  \ast_{d}\eta\right)  \left(  \frac{d\left(  x,y\right)
}{2\pi}\right)  .
\end{align*}
Using (\ref{Transplantation}) and the fast decay at infinity of $\mathcal{F}%
_{d}\psi$, for any positive $M$ there exist positive constants $C$ and $G$
such that for every $\rho\geq0$
\begin{align*}
&  \left\vert \mathcal{F}_{d-2\nu}\left(  H\left(  \frac{\cdot}{\lambda_{X}%
}\right)  \ast_{d}\eta\right)  \left(  \rho\right)  \right\vert =\lambda
_{X}^{d}\left\vert \mathcal{F}_{d-2\nu}\left(  \mathcal{F}_{d}\left(  \left(
\mathcal{F}_{d}\psi\right)  ^{2}\left(  \lambda_{X}\cdot\right)
\mathcal{F}_{d}\eta\right)  \right)  \left(  \rho\right)  \right\vert \\
&  =c_{d,d-2\nu}\lambda_{X}^{d}\int_{\rho}^{+\infty}\left(  r^{2}-\rho
^{2}\right)  ^{\nu-1}r\left(  \mathcal{F}_{d}\psi\left(  \lambda_{X}r\right)
\right)  ^{2}\mathcal{F}_{d}\eta\left(  r\right)  dr\\
&  \leq C\lambda_{X}^{d}\int_{\rho}^{+\infty}\left(  r^{2}-\rho^{2}\right)
^{\nu-1}r\left(  \mathcal{F}_{d}\psi\left(  \lambda_{X}r\right)  \right)
^{2}dr\\
&  \leq C\lambda_{X}^{d}\int_{\rho}^{+\infty}\left(  r^{2}-\rho^{2}\right)
^{\nu-1}r\frac{C}{\left(  1+\lambda_{X}r\right)  ^{G}}dr\\
&  \leq C\lambda_{X}^{d-2\nu}\int_{\lambda_{X}\rho}^{+\infty}\left(
u^{2}-\left(  \lambda_{X}\rho\right)  ^{2}\right)  ^{\nu-1}u\frac{C}{\left(
1+u\right)  ^{G}}du\leq\frac{C\lambda_{X}^{d-2\nu}}{\left(  1+\left\vert
\lambda_{X}\rho\right\vert \right)  ^{M}}.
\end{align*}
Therefore, using the symmetry of $\Omega_{\nu}\left(  x,y\right)  $, for any
integer $\nu$ with $1\leq\nu<d/2$ we obtain%
\begin{align*}
&  \left\vert \sum_{r=1}^{R}\sum_{j=1}^{K_{r}}\sum_{s=1}^{R}\sum_{i=1}^{K_{s}%
}a_{r,j}a_{s,i}\alpha_{\nu}\left(  x_{r,j},x_{s,i}\right)  \Omega_{\nu}\left(
x_{r,j},x_{s,i}\right)  \right\vert \\
\leq &  2C\sum_{r=1}^{R}\sum_{j=1}^{K_{r}}\sum_{s=r}^{R}\sum_{i=1}^{K_{s}%
}a_{r,j}a_{s,i}\left\vert \Omega_{\nu}\left(  x_{r,j},x_{s,i}\right)
\right\vert \\
\leq &  C\sum_{r=1}^{R}\sum_{j=1}^{K_{r}}\sum_{s=r}^{R}\sum_{i=1}^{K_{s}%
}a_{r,j}a_{s,i}\lambda_{X}^{d-2\nu}\frac{1}{\left(  1+\lambda_{X}d\left(
x_{r,j},x_{s,i}\right)  \right)  ^{M}}.
\end{align*}
In order to estimate the above sum recall that every region $\mathcal{B}_{r}$
is contained in a ball centered at a point $z_{r}\in\mathcal{B}_{r}$ of radius
$c_{2}Y^{-1/d}$ and let $c_{3}=10c_{2}$. For every fixed $r=1,\ldots,R$ we
will consider separately the contribution of those values of $s$ for which the
$\mathcal{B}_{s}$ is near $\mathcal{B}_{r}$, in the sense that $\mathcal{B}%
_{s}$ is contained in the ball centered at $z_{r}$ and with radius
$c_{3}Y^{-1/d}$, and the contribution of the remaining values of $s$, for
which we will say that $\mathcal{B}_{s}$ is far from $\mathcal{B}_{r}$. Notice
that there are at most
\[
\frac{\left\vert B\left(  z_{r},c_{3}Y^{-1/d}\right)  \right\vert }{Y^{-1}%
}\leq\frac{C\left(  c_{3}Y^{-1/d}\right)  ^{d}}{Y^{-1}}\leq Cc_{3}^{d}%
\]
regions $\mathcal{B}_{s}$ near $\mathcal{B}_{r}$. Thus, using again that
$\lambda_{X}\sim X^{1/d}$ and that for $r\leqslant s$ we have $\sum
_{j=1}^{K_{r}}a_{r,j}\geq\sum_{i=1}^{K_{s}}a_{s,i}$, we obtain
\begin{align*}
&  \sum_{r=1}^{R}\sum_{j=1}^{K_{r}}\sum_{s=r}^{R}\sum_{i=1}^{K_{s}}%
a_{r,j}a_{s,i}\lambda_{X}^{d-2\nu}\frac{1}{\left(  1+\lambda_{X}d\left(
x_{r,j},x_{s,i}\right)  \right)  ^{M}}\\
\leq &  CX^{1-2\nu/d}\sum_{r=1}^{R}\sum_{\substack{s=r\,\\B_{s}%
\,\text{near\thinspace}B_{r}}}^{R}\sum_{j=1}^{K_{r}}a_{r,j}\sum_{i=1}^{K_{s}%
}a_{s,i}\\
&  +CX^{1-2\nu/d}\sum_{r=1}^{R}\sum_{\substack{s=r\,\\B_{s}\text{\thinspace
far from\thinspace}B_{r}}}^{R}\sum_{j=1}^{K_{r}}a_{r,j}\sum_{i=1}^{K_{s}%
}a_{s,i}\left(  \lambda_{X}d\left(  x_{r,j},x_{s,i}\right)  \right)  ^{-M}\\
\leq &  CX^{1-2\nu/d}\sum_{r=1}^{R}\left(  \sum_{j=1}^{K_{r}}a_{r,j}\right)
^{2}\\
&  +CX^{1-2\nu/d}\sum_{r=1}^{R-1}\sum_{\substack{s=r+1\\B_{s}\text{\thinspace
far from\thinspace}B_{r}}}^{R}\sum_{j=1}^{K_{r}}a_{r,j}\sum_{i=1}^{K_{s}%
}a_{s,i}\left(  X^{1/d}d\left(  x_{r,j},x_{s,i}\right)  \right)  ^{-M}.
\end{align*}

Using again that for $r\leq s$ we have $\sum_{j=1}^{K_{r}}a_{r,j}\geq\sum
_{i=1}^{K_{s}}a_{s,i}$,
\begin{align*}
&  \sum_{r=1}^{R-1}\sum_{\substack{s=r+1\\B_{s}\text{\thinspace far
from\thinspace}B_{r}}}^{R}\sum_{j=1}^{K_{r}}a_{r,j}\sum_{i=1}^{K_{s}}%
a_{s,i}\left(  X^{1/d}d\left(  x_{r,j},x_{s,i}\right)  \right)  ^{-M}\\
= &  \sum_{r=1}^{R-1}\sum_{j=1}^{K_{r}}a_{r,j}\sum_{\ell=0}^{\infty}%
\sum_{\substack{s>r:\\2^{\ell-1}c_{3}Y^{-1/d}\leq d\left(  z_{r},z_{s}\right)
\leq2^{\ell}c_{3}Y^{-1/d}}}\sum_{i=1}^{K_{s}}a_{s,i}\left(  X^{1/d}d\left(
x_{r,j},x_{s,i}\right)  \right)  ^{-M}\\
\leq &  C\sum_{r=1}^{R-1}\sum_{j=1}^{K_{r}}a_{r,j}\sum_{\ell=0}^{\infty
}2^{-\ell M}\sum_{\substack{s>r:\\d\left(  z_{r},z_{s}\right)  \leq2^{\ell}c_{3}Y^{-1/d}}}\sum_{i=1}^{K_{s}}a_{s,i}\\
\leq &  C\sum_{r=1}^{R-1}\sum_{j=1}^{K_{r}}a_{r,j}\sum_{\ell=0}^{\infty
}2^{-\ell M}\frac{\left(  2^{\ell}Y^{-1/d}\right)  ^{d}}{Y^{-1}}\sum
_{j=1}^{K_{r}}a_{r,j}\\
\leq &  C\sum_{r=1}^{R-1}\left(  \sum_{j=1}^{K_{r}}a_{r,j}\right)  ^{2}%
\sum_{\ell=0}^{\infty}2^{-\ell\left(  M-d\right)  }\leq C\sum_{r=1}%
^{R-1}\left(  \sum_{j=1}^{K_{r}}a_{r,j}\right)  ^{2}.
\end{align*}

\end{proof}

\begin{lemma}
There exist $C>0$ and $X_{0}$ such that for every $X>X_{0}$%
\[
\left\vert \sum_{r=1}^{R}\sum_{j=1}^{K_{r}}\sum_{s=1}^{R}\sum_{i=1}^{K_{s}%
}a_{r,j}a_{s,i}F_{3}\left(  x_{r,j},x_{s,i}\right)  \right\vert \leq
C\sum_{r=1}^{R}\left(  \sum_{j=1}^{K_{r}}a_{r,j}\right)  ^{2}.
\]

\end{lemma}

\begin{proof}
We will show that for every integer $\nu\geq d/2,$%
\begin{align*}
&  \left\vert \sum_{r=1}^{R}\sum_{j=1}^{K_{r}}\sum_{s=1}^{R}\sum_{i=1}^{K_{s}%
}a_{r,j}a_{s,i}\alpha_{\nu}\left(  x_{r,j},x_{s,i}\right)  \Omega_{\nu}\left(
x_{r,j},x_{s,i}\right)  \right\vert \\
\leq &  CX^{1-2\nu/d}\sum_{r=1}^{R}\left(  \sum_{j=1}^{K_{r}}a_{r,j}\right)
^{2}.
\end{align*}
Observe that for $\nu\geq d/2$, the distribution $\partial_{t}\left(  E_{\nu
}-\check{E}_{\nu}\right)  \left(  t,d\left(  x,y\right)  \right)  $ can be
identified with the locally integrable function%
\[
C_{\nu}|t|\left(  t^{2}-d\left(  x,y\right)  ^{2}\right)  _{+}^{\nu-1+\left(
1-d\right)  /2},
\]
for an appropriate value of $C_{\nu}$. Therefore, using the symmetry of
$\Omega_{\nu}\left(  x,y\right)  $,%
\begin{align*}
&  \left\vert \sum_{r=1}^{R}\sum_{j=1}^{K_{r}}\sum_{s=1}^{R}\sum_{i=1}^{K_{s}%
}a_{r,j}a_{s,i}\alpha_{\nu}\left(  x_{r,j},x_{s,i}\right)  \Omega_{\nu}\left(
x_{r,j},x_{s,i}\right)  \right\vert \\
\leq &  C\sum_{r=1}^{R}\sum_{j=1}^{K_{r}}\sum_{s=r}^{R}\sum_{i=1}^{K_{s}%
}a_{r,j}a_{s,i}\int_{d\left(  x_{r,j},x_{s,i}\right)  }^{+\infty}t\left(
t^{2}-d\left(  x_{r,j},x_{s,i}\right)  ^{2}\right)  ^{\nu-1+\left(
1-d\right)  /2}\mathcal{C}^{-1}\widetilde{H}\left(  t\right)  dt,
\end{align*}
where we use the fact that $\mathcal{C}^{-1}\widetilde{H}\left(  t\right)
\geq0$, by Lemma \ref{support}.

\noindent Assume first that $d=1$ and let $D=d\left(  x_{r,j},x_{s,i}\right)
$, then%
\begin{align*}
&  \int_{D}^{+\infty}t\left(  t^{2}-D^{2}\right)  ^{\nu-1+\left(  1-d\right)
/2}\mathcal{C}^{-1}\widetilde{H}\left(  t\right)  dt\\
&  =\int_{D}^{+\infty}t\left(  t^{2}-D^{2}\right)  ^{\nu-1}\frac{1}{\pi
}\mathcal{F}_{1}\widetilde{H}\left(  \frac{t}{2\pi}\right)  dt\\
&  =\lambda_{X}\int_{D}^{+\infty}t\left(  t^{2}-D^{2}\right)  ^{\nu-1}\frac
{1}{\pi}\left(  \mathcal{F}_{1}\psi\left(  \lambda_{X}\frac{t}{2\pi}\right)
\right)  ^{2}\mathcal{F}_{1}\eta\left(  \frac{t}{2\pi}\right)  dt\\
&  \leq c\lambda_{X}\int_{\frac{D}{2\pi}}^{+\infty}u^{\nu}\left(  u-\frac
{D}{2\pi}\right)  ^{\nu-1}\left(  \mathcal{F}_{1}\psi\left(  \lambda
_{X}u\right)  \right)  ^{2}\mathcal{F}_{1}\eta\left(  u\right)  du.
\end{align*}
A similar estimate can be obtained for $d\geq2$. Indeed, by formula
(\ref{Transplantation})%
\begin{align*}
\mathcal{C}^{-1}\widetilde{H}\left(  t\right)   &  =\frac{2}{\pi}%
\mathcal{C}\widetilde{H}\left(  t\right)  \\
&  =\frac{1}{\pi}\mathcal{F}_{1}\mathcal{F}_{d}\mathcal{F}_{d}\widetilde
{H}\left(  \frac{t}{2\pi}\right)  \\
&  =c\int_{\frac{t}{2\pi}}^{+\infty}\left(  u^{2}-\left(  \frac{t}{2\pi
}\right)  ^{2}\right)  ^{\left(  d-1\right)  /2-1}u\mathcal{F}_{d}%
\widetilde{H}\left(  u\right)  du\\
&  =c\lambda_{X}^{d}\int_{\frac{t}{2\pi}}^{+\infty}\left(  u^{2}-\left(
\frac{t}{2\pi}\right)  ^{2}\right)  ^{\left(  d-1\right)  /2-1}u\left(
\mathcal{F}_{d}\psi\left(  \lambda_{X}u\right)  \right)  ^{2}\mathcal{F}%
_{d}\eta\left(  u\right)  du,
\end{align*}
so that, by Fubini's theorem,%
\begin{align*}
&  \int_{D}^{+\infty}t\left(  t^{2}-D^{2}\right)  ^{\nu-1+\left(  1-d\right)
/2}\mathcal{C}^{-1}\widetilde{H}\left(  t\right)  dt\\
= &  c\lambda_{X}^{d}\int_{\frac{D}{2\pi}}^{+\infty}u\left(  \mathcal{F}%
_{d}\psi\left(  \lambda_{X}u\right)  \right)  ^{2}\mathcal{F}_{d}\eta\left(
u\right)  \\
&  \hspace{0.7cm}\times\left(  \int_{D}^{2\pi u}2t\left(  t^{2}-D^{2}\right)
^{\nu-d/2-1/2}\left(  u^{2}-\left(  \frac{t}{2\pi}\right)  ^{2}\right)
^{\left(  d-3\right)  /2}dt\right)  du.
\end{align*}
Since
\begin{align*}
&  \int_{D}^{U}t\left(  t^{2}-D^{2}\right)  ^{\nu-d/2-1/2}\left(  U^{2}%
-t^{2}\right)  ^{\left(  d-3\right)  /2}dt\\
\leq &  CU^{\left(  d-3\right)  /2}\int_{D}^{U}t^{\nu-d/2+1/2}\left(
t-D\right)  ^{\nu-d/2-1/2}\left(  U-t\right)  ^{\left(  d-3\right)  /2}dt\\
\leq &  CU^{\left(  d-3\right)  /2}U^{\nu-d/2+1/2}\\
&  \hspace{0.7cm}\times\int_{0}^{1}\left(  z\left(  U-D\right)  \right)
^{\nu-d/2-1/2}\left(  \left(  1-z\right)  \left(  U-D\right)  \right)
^{\left(  d-3\right)  /2}\left(  U-D\right)  dz\\
= &  CU^{\nu-1}\left(  U-D\right)  ^{\nu-1}\int_{0}^{1}z^{\nu-d/2-1/2}\left(
1-z\right)  ^{\left(  d-3\right)  /2}dz,
\end{align*}
we obtain%
\begin{align*}
&  \int_{D}^{+\infty}t\left(  t^{2}-D^{2}\right)  ^{\nu-1+\left(  1-d\right)
/2}\mathcal{C}^{-1}\widetilde{H}\left(  t\right)  dt\\
\leq &  c\lambda_{X}^{d}\int_{\frac{D}{2\pi}}^{+\infty}u^{\nu}\left(
u-\frac{D}{2\pi}\right)  ^{\nu-1}\left(  \mathcal{F}_{d}\psi\left(
\lambda_{X}u\right)  \right)  ^{2}\mathcal{F}_{d}\eta\left(  u\right)  du.
\end{align*}
Thus, for all $d\geq1$,%
\begin{align*}
&  \int_{D}^{+\infty}t\left(  t^{2}-D^{2}\right)  ^{\nu-1+\left(  1-d\right)
/2}\mathcal{C}^{-1}\widetilde{H}\left(  t\right)  dt\\
\leq &  c\lambda_{X}^{d}\int_{\frac{D}{2\pi}}^{+\infty}u^{\nu}\left(
u-\frac{D}{2\pi}\right)  ^{\nu-1}\left(  \mathcal{F}_{d}\psi\left(
\lambda_{X}u\right)  \right)  ^{2}\mathcal{F}_{d}\eta\left(  u\right)  du\\
\leq &  c\lambda_{X}^{d-1}\int_{\lambda_{X}\frac{D}{2\pi}}^{+\infty}\left(
\frac{v}{\lambda_{X}}\right)  ^{\nu}\left(  \frac{v}{\lambda_{X}}-\frac
{D}{2\pi}\right)  ^{\nu-1}\left(  \mathcal{F}_{d}\psi\left(  v\right)
\right)  ^{2}\mathcal{F}_{d}\eta\left(  \frac{v}{\lambda_{x}}\right)  dv\\
\leq &  c\lambda_{X}^{d-2\nu}\int_{\lambda_{X}\frac{D}{2\pi}}^{+\infty}v^{\nu
}\left(  v-\frac{D\lambda_{X}}{2\pi}\right)  ^{\nu-1}\left(  \mathcal{F}%
_{d}\psi\left(  v\right)  \right)  ^{2}dv\\
\leq &  c\lambda_{X}^{d-2\nu}\int_{\lambda_{X}\frac{D}{2\pi}}^{+\infty}%
v^{2\nu-1}\left(  \mathcal{F}_{d}\psi\left(  v\right)  \right)  ^{2}dv\\
\leq &  c\lambda_{X}^{d-2\nu}\frac{1}{\left(  1+\lambda_{X}D\right)  ^{M}}.
\end{align*}
Finally, arguing as in the previous lemma,%
\begin{align*}
&  \sum_{r=1}^{R}\sum_{j=1}^{K_{r}}\sum_{s=r}^{R}\sum_{i=1}^{K_{s}}%
a_{r,j}a_{s,i}\int_{d\left(  x_{r,j},x_{s,i}\right)  }^{+\infty}t\left(
t^{2}-d\left(  x_{r,j},x_{s,i}\right)  ^{2}\right)  ^{\nu-1+\left(
1-d\right)  /2}\mathcal{C}^{-1}\widetilde{H}\left(  t\right)  dt\\
&  \leq C\sum_{r=1}^{R}\sum_{j=1}^{K_{r}}\sum_{s=r}^{R}\sum_{i=1}^{K_{s}%
}a_{r,j}a_{s,i}\lambda_{X}^{d-2\nu}\frac{1}{\left(  1+\lambda_{X}d\left(
x_{r,j},x_{s,i}\right)  \right)  ^{M}}\\
&  \leq CX^{1-2\nu/d}\sum_{r=1}^{R}\left(  \sum_{j=1}^{K_{r}}a_{r,j}\right)
^{2}.
\end{align*}

\end{proof}

\begin{lemma}
There exist $C>0$ and $X_{0}>0$ such that for every $X>X_{0}$%
\[
\left\vert \sum_{r=1}^{R}\sum_{j=1}^{K_{r}}\sum_{s=1}^{R}\sum_{i=1}^{K_{s}%
}a_{r,j}a_{s,i}F_{4}\left(  x_{r,j},x_{s,i}\right)  \right\vert \leq
C\sum_{r=1}^{R}\left(  \sum_{j=1}^{K_{r}}a_{r,j}\right)  ^{2}.
\]

\end{lemma}

\begin{proof}
Recall that for every $x,y\in\mathcal{M}$,
\[
F_{4}\left(  x,y\right)  =\int_{0}^{\varepsilon}R_{Q}\left(  t,x,y\right)
\mathcal{C}^{-1}\widetilde{H}\left(  t\right)  dt.
\]
As before, if $d=1$, then%
\[
\mathcal{C}^{-1}\widetilde{H}\left(  t\right)  =\frac{1}{\pi}\mathcal{F}%
_{1}\widetilde{H}\left(  \frac{t}{2\pi}\right)  =\frac{\lambda_{X}}{\pi
}\left(  \mathcal{F}_{1}\psi\left(  \lambda_{X}\frac{t}{2\pi}\right)  \right)
^{2}\mathcal{F}_{1}\eta\left(  \frac{t}{2\pi}\right)
\]
and, by Theorem \ref{sogge},%
\begin{align*}
\left\vert F_{4}\left(  x,y\right)  \right\vert  &  \leq\frac{\lambda_{X}}%
{\pi}\int_{0}^{\varepsilon}\left\vert R_{Q}\left(  t,x,y\right)  \right\vert
\left(  \mathcal{F}_{1}\psi\left(  \lambda_{X}\frac{t}{2\pi}\right)  \right)
^{2}\mathcal{F}_{1}\eta\left(  \frac{t}{2\pi}\right)  dt\\
&  \leq c\lambda_{X}^{1-2Q-2}\int_{0}^{+\infty}t^{2Q+1}\left(  \mathcal{F}%
_{1}\psi\left(  t\right)  \right)  ^{2}dt,
\end{align*}

A similar estimate holds for $d\geq2$. Indeed, as in the previous lemma,%
\[
\mathcal{C}^{-1}\widetilde{H}\left(  t\right)  =c\lambda_{X}^{d}\int_{\frac
{t}{2\pi}}^{+\infty}\left(  u^{2}-\left(  \frac{t}{2\pi}\right)  ^{2}\right)
^{\left(  d-1\right)  /2-1}u\left(  \mathcal{F}_{d}\psi\left(  \lambda
_{X}u\right)  \right)  ^{2}\mathcal{F}_{d}\eta\left(  u\right)  du
\]
so that, again by Theorem \ref{sogge},
\begin{align*}
&  \left\vert F_{4}\left(  x,y\right)  \right\vert \\
= & c \left\vert \lambda_{X}^{d}\int_{0}^{\varepsilon}R_{Q}\left(
t,x,y\right)  \int_{\frac{t}{2\pi}}^{+\infty}\left(  u^{2}-\left(  \frac
{t}{2\pi}\right)  ^{2}\right)  ^{\left(  d-1\right)  /2-1}u\left(
\mathcal{F}_{d}\psi\left(  \lambda_{X}u\right)  \right)  ^{2}\mathcal{F}%
_{d}\eta\left(  u\right)  dudt\right\vert \\
\leq & c  \lambda_{X}^{d}\int_{0}^{+\infty}u\left(  \mathcal{F}_{d}\psi\left(
\lambda_{X}u\right)  \right)  ^{2}\mathcal{F}_{d}\eta\left(  u\right)
\\
&  \hspace{1cm}\times\int_{0}^{\min\left(  2\pi u,\varepsilon\right)  }\left\vert
R_{Q}\left(  t,x,y\right)  \right\vert \left(  u^{2}-\left(  \frac{t}{2\pi
}\right)  ^{2}\right)  ^{\left(  d-1\right)  /2-1}dtdu\\
\leq &  c\lambda_{X}^{d}\int_{0}^{+\infty}u\left(  \mathcal{F}_{d}\psi\left(
\lambda_{X}u\right)  \right)  ^{2}\mathcal{F}_{d}\eta\left(  u\right)
\int_{0}^{2\pi u}t^{2Q+2-d}\left(  u^{2}-\left(  \frac{t}{2\pi}\right)
^{2}\right)  ^{\left(  d-1\right)  /2-1}dtdu\\
\leq &  c\lambda_{X}^{d}\int_{0}^{+\infty}\left(  \mathcal{F}_{d}\psi\left(
\lambda_{X}u\right)  \right)  ^{2}\mathcal{F}_{d}\eta\left(  u\right)
u^{2Q+1}du.\\
\leq &  c\lambda_{X}^{d-2Q-2}\int_{0}^{+\infty}\left(  \mathcal{F}_{d}%
\psi\left(  v\right)  \right)  ^{2}v^{2Q+1}dv.
\end{align*}
Finally, for all $d\geq1$, since $R\leq Y=\kappa X$,%
\begin{align*}
&  \left\vert \sum_{r=1}^{R}\sum_{j=1}^{K_{r}}\sum_{s=1}^{R}\sum_{i=1}^{K_{s}%
}a_{r,j}a_{s,i}\int_{0}^{\varepsilon}R_{Q}\left(  t,x_{r,j},x_{s,i}\right)
\mathcal{C}^{-1}\widetilde{H}\left(  t\right)  dt\right\vert \\
\leq &  c\lambda_{X}^{d-2Q-2}\sum_{r=1}^{R}\sum_{j=1}^{K_{r}}\sum_{s=1}%
^{R}\sum_{i=1}^{K_{s}}a_{r,j}a_{s,i}\\
\leq &  cX^{1-\left(  2Q+2\right)  /d}\sum_{r=1}^{R}\sum_{s=1}^{R}\sum
_{j=1}^{K_{r}}a_{r,j}\sum_{i=1}^{K_{s}}a_{s,i}\\
\leq &  cX^{1-\left(  2Q+2\right)  /d}\sum_{r=1}^{R}\sum_{s=r}^{R}\left(
\sum_{j=1}^{K_{r}}a_{r,j}\right)  ^{2}\\
\leq &  cX^{1-\left(  2Q+2\right)  /d}\sum_{r=1}^{R}\left(  R-r+1\right)
\left(  \sum_{j=1}^{K_{r}}a_{r,j}\right)  ^{2}\\
\leq &  c\kappa X^{2-\left(  2Q+2\right)  /d}\sum_{r=1}^{R}\left(  \sum
_{j=1}^{K_{r}}a_{r,j}\right)  ^{2},
\end{align*}
and since $Q>d+3$, the exponent $2-\left(  2Q+2\right)  /d$ is negative and
the result follows.
\end{proof}

\begin{lemma}
There exist $C>0$ and $X_{0}>0$ such that for every $X>X_{0}$%
\[
\left\vert \sum_{r=1}^{R}\sum_{j=1}^{K_{r}}\sum_{s=1}^{R}\sum_{i=1}^{K_{s}%
}a_{r,j}a_{s,i}F_{5}\left(  x_{r,j},x_{s,i}\right)  \right\vert \leq
C\sum_{r=1}^{R}\left(  \sum_{j=1}^{K_{r}}a_{r,j}\right)  ^{2}.
\]

\end{lemma}

\begin{proof}
We need to estimate%
\begin{align*}
&  \left\vert \sum_{r=1}^{R}\sum_{j=1}^{K_{r}}\sum_{s=1}^{R}\sum_{i=1}^{K_{s}%
}a_{r,j}a_{s,i}\sum_{m=0}^{+\infty}\left(  H\left(  \frac{\lambda_{m}}%
{\lambda_{X}}\right)  -\widetilde{H}\left(  \lambda_{m}\right)  \right)
\varphi_{m}\left(  x_{r,j}\right)  \overline{\varphi_{m}\left(  x_{s,i}%
\right)  }\right\vert \\
\leq &  2\sum_{r=1}^{R}\sum_{j=1}^{K_{r}}\sum_{s=r}^{R}\sum_{i=1}^{K_{s}}%
\sum_{m=0}^{+\infty}a_{r,j}a_{s,i}\left\vert H\left(  \frac{\lambda_{m}%
}{\lambda_{X}}\right)  -\widetilde{H}\left(  \lambda_{m}\right)  \right\vert
\left\vert \varphi_{m}\left(  x_{r,j}\right)  \right\vert \left\vert
\varphi_{m}\left(  x_{s,i}\right)  \right\vert .
\end{align*}
Let us first study the term
\[
I\left(  x\right)  =H\left(  \frac{x}{\lambda_{X}}\right)  -\widetilde
{H}\left(  x\right)  =\int_{\mathbb{R}^{d}}\left[  H\left(  \frac{x}%
{\lambda_{X}}\right)  -H\left(  \frac{x-y}{\lambda_{X}}\right)  \right]
\eta\left(  y\right)  dy.
\]
Since $\eta\left(  y\right)  $ has rapid decay at infinity and $H\left(
x\right)  $ is supported in $B\left(  0,1\right)  $ , if $\left\vert
x\right\vert \geq2\lambda_{X}$ we have%
\begin{align*}
\left\vert I\left(  x\right)  \right\vert  &  \leq\int_{\mathbb{R}^{d}%
}\left\vert H\left(  \frac{x-y}{\lambda_{X}}\right)  \eta\left(  y\right)
\right\vert dy\leq\int_{\left\{  \left\vert x-y\right\vert \leq\lambda
_{X}\right\}  }\left\vert H\left(  \frac{x-y}{\lambda_{X}}\right)  \eta\left(
y\right)  \right\vert dy\\
&  \leq c\int_{\left\{  \left\vert y\right\vert \geq\left\vert x\right\vert
-\lambda_{X}\right\}  }\left\vert \eta\left(  y\right)  \right\vert dy\leq
C\left(  1+\left\vert x\right\vert -\lambda_{X}\right)  ^{-M}.
\end{align*}
Assume $\left\vert x\right\vert <2\lambda_{X}$. By Taylor's theorem with
integral reminder we can write%
\begin{align*}
&  H\left(  \frac{x}{\lambda_{X}}-\frac{y}{\lambda_{X}}\right)  \\
= &  H\left(  \frac{x}{\lambda_{X}}\right)  +\sum_{1\leq\left\vert
\alpha\right\vert \leq M-1}\frac{1}{\alpha!}\frac{\partial^{\left\vert
\alpha\right\vert }H}{\partial x^{\alpha}}\left(  \frac{x}{\lambda_{X}%
}\right)  \left(  -\frac{y}{\lambda_{X}}\right)  ^{\alpha}\\
&  +\sum_{\left\vert \alpha\right\vert =M}\frac{M}{\alpha!}\left(  -\frac
{y}{\lambda_{X}}\right)  ^{\alpha}\int_{0}^{1}\left(  1-t\right)  ^{M-1}%
\frac{\partial^{\left\vert \alpha\right\vert }H}{\partial x^{\alpha}}\left(
\frac{x}{\lambda_{X}}-t\frac{y}{\lambda_{X}}\right)  dt
\end{align*}
so that%
\[
H\left(  \frac{x}{\lambda_{X}}-\frac{y}{\lambda_{X}}\right)  =H\left(
\frac{x}{\lambda_{X}}\right)  +\sum_{1\leq\left\vert \alpha\right\vert \leq
M-1}\frac{1}{\alpha!}\frac{\partial^{\left\vert \alpha\right\vert }H}{\partial
x^{\alpha}}\left(  \frac{x}{\lambda_{X}}\right)  \left(  -\frac{y}{\lambda
_{X}}\right)  ^{\alpha}+O\left(  \left\vert -\frac{y}{\lambda_{X}}\right\vert
^{M}\right)  .
\]
It follows that%
\begin{align*}
I\left(  x\right)   &  =\int_{\mathbb{R}^{d}}\left[  H\left(  \frac{x}%
{\lambda_{X}}\right)  -H\left(  \frac{x-y}{\lambda_{X}}\right)  \right]
\eta\left(  y\right)  dy\\
&  =-\sum_{1\leq\left\vert \alpha\right\vert \leq M-1}\frac{1}{\alpha!}%
\frac{\partial^{\left\vert \alpha\right\vert }H}{\partial x^{\alpha}}\left(
\frac{x}{\lambda_{X}}\right)  \int_{\mathbb{R}^{d}}\left(  -\frac{y}%
{\lambda_{X}}\right)  ^{\alpha}\eta\left(  y\right)  dy+\int_{\mathbb{R}^{d}%
}O\left(  \frac{\left\vert y\right\vert ^{M}}{\lambda_{X}^{M}}\right)
\eta\left(  y\right)  dy
\end{align*}
and since
\[
\int_{\mathbb{R}^{d}}y^{\alpha}\eta\left(  y\right)  dy=\mathcal{F}_{d}\left(
y^{\alpha}\eta\left(  y\right)  \right)  \left(  0\right)  =\left(  -2\pi
i\right)  ^{-\left\vert \alpha\right\vert }\frac{\partial^{\left\vert
\alpha\right\vert }\mathcal{F}_{d}\eta}{\partial\xi^{\alpha}}\left(  0\right)
=0
\]
we obtain%
\[
\left\vert I\left(  x\right)  \right\vert \leq c\lambda_{X}^{-M}.
\]
Using Hormander's estimates on the $L^{\infty}$ norm of the eigenfunctions (see \cite[(3.2.2), page 48]{sogge})
\[
\left\Vert \varphi_{m}\right\Vert _{\infty}\leq C\left(  1+\lambda
_{m}\right)  ^{\frac{d-1}{2}},\]
we obtain
\begin{align*}
&  \sum_{r=1}^{R}\sum_{j=1}^{K_{r}}\sum_{s=r}^{R}\sum_{i=1}^{K_{s}}\sum
_{m=0}^{+\infty}a_{r,j}a_{s,i}\left\vert H\left(  \frac{\lambda_{m}}%
{\lambda_{X}}\right)  -\widetilde{H}\left(  \lambda_{m}\right)  \right\vert
\left\vert \varphi_{m}\left(  x_{r,j}\right)  \right\vert \left\vert
\varphi_{m}\left(  x_{s,i}\right)  \right\vert \\
\leq &  c\lambda_{X}^{-M}\sum_{\lambda_{m}\leq2\lambda_{X}}\left(
1+\lambda_{m}\right)  ^{d-1}\sum_{r=1}^{R}\sum_{s=r}^{R}\sum_{j=1}^{K_{r}%
}a_{r,j}\sum_{i=1}^{K_{s}}a_{s,i}\\
&  +c\sum_{\lambda_{m}\geq2\lambda_{X}}\left(  1+\lambda_{m}-\lambda
_{X}\right)  ^{-M}\lambda_{m}^{d-1}\sum_{r=1}^{R}\sum_{s=r}^{R}\sum
_{j=1}^{K_{r}}a_{r,j}\sum_{i=1}^{K_{s}}a_{s,i}\\
\leq &  c\left(  \lambda_{X}^{-M}\sum_{\lambda_{m}\leq2\lambda_{X}}\left(
1+\lambda_{m}\right)  ^{d-1}+\sum_{\lambda_{m}\geq2\lambda_{X}}\lambda
_{m}^{-M+d-1}\right)  \sum_{r=1}^{R}\left(  R-r+1\right)  \left(  \sum
_{j=1}^{K_{r}}a_{r,j}\right)  ^{2}\\
\leq &  c\kappa X\left(  \lambda_{X}^{-M}\sum_{\lambda_{m}\leq2\lambda_{X}%
}\left(  1+\lambda_{m}\right)  ^{d-1}+\sum_{\lambda_{m}\geq2\lambda_{X}%
}\lambda_{m}^{-M+d-1}\right)  \sum_{r=1}^{R}\left(  \sum_{j=1}^{K_{r}}%
a_{r,j}\right)  ^{2}.
\end{align*}
By Weyl's estimates, which say that the number of eigenvalues $\lambda_{m}%
^{2}\leq T^{2}$ is $\ $asymptotic to $cT^{d}$,%
\begin{align*}
&  \lambda_{X}^{-M}\sum_{\lambda_{m}\leq2\lambda_{X}}\left(  1+\lambda
_{m}\right)  ^{d-1}+\sum_{\lambda_{m}\geq2\lambda_{X}}\lambda_{m}^{-M+d-1}\\
\leq &  c\lambda_{X}^{-M}\lambda_{X}^{d}\lambda_{X}^{d-1}+\sum_{k=1}^{+\infty
}\sum_{2^{k}\lambda_{X}\leq\lambda_{m}\leq2^{k+1}\lambda_{X}}\lambda
_{m}^{-M+d-1}\\
\leq &  c\lambda_{X}^{-M+2d-1}+c\sum_{k=1}^{+\infty}2^{d\left(  k+1\right)
}\lambda_{X}^{d}\left(  2^{k}\lambda_{X}\right)  ^{-M+d-1}\\
\leq &  c\lambda_{X}^{-M+2d-1}+c\lambda_{X}^{-M+2d-1}\sum_{k=1}^{+\infty
}2^{\left(  2d-M-1\right)  k}%
\end{align*}
and taking $M$ such that $-M+2d-1<-d$ gives the result.
\end{proof}

\section{Final remarks}

A simple consequence of Theorem \ref{main} is the following estimate on the
maximum degree $X$ of linear combinations of eigenfunctions of the Laplacian
up to the eigenvalue $\lambda_{X}$ that a quadrature rule can integrate
exactly. This is a well known result for equal weights, see e.g.
\cite[Proposition 1]{GG}, or \cite[Theorem 2]{Steiner} where one can find an
estimate of the constant $C$ that depends only on the dimension of the
manifold. See also \cite[Theorem 1]{Steiner} for a result with general weights.

\begin{corollary}
Let $X$ be a positive integer and assume there exist points $\left\{  x\left(
j\right)  \right\}  _{j=1}^{N}$ and weights $\left\{  a_{j}\right\}
_{j=1}^{N}$ such that for every polynomial
\[
P\left(  x\right)  =\sum_{m=0}^{X}c_{m}\varphi_{m}\left(  x\right)
\]
we have%
\[
\int_{\mathcal{M}}P\left(  x\right)  dx=\sum_{j=1}^{N}a_{j}P\left(  x\left(
j\right)  \right)  .
\]
Then there exists a constant $C>0$ independent of $X$ and $N$ such that
\[
1\geq CX\sum_{j=1}^{N}a_{j}^{2}.
\]
In particular%
\[
CX\leqslant N.
\]

\end{corollary}

\begin{proof}
Since $\varphi_{0}\left(  x\right)  \equiv1$ we must have $\sum_{i=1}^{N}%
a_{i}=1$. Let
\[
P\left(  x\right)  =\sum_{m=0}^{X}\sum_{i=1}^{N}a_{i}\overline{\varphi
_{m}\left(  x\left(  i\right)  \right)  }\varphi_{m}\left(  x\right)  ,
\]
then%
\begin{align*}
\int_{\mathcal{M}}P\left(  x\right)  dx  &  =\int_{\mathcal{M}}\sum_{m=0}%
^{X}\sum_{i=1}^{N}a_{i}\overline{\varphi_{m}\left(  x\left(  i\right)
\right)  }\varphi_{m}\left(  x\right)  dx\\
&  =\sum_{m=0}^{X}\sum_{i=1}^{N}a_{i}\overline{\varphi_{m}\left(  x\left(
i\right)  \right)  }\int_{\mathcal{M}}\varphi_{m}\left(  x\right)
dx=\sum_{i=1}^{N}a_{i}=1.
\end{align*}
On the other hand%
\begin{align*}
\sum_{j=1}^{N}a_{j}P\left(  x\left(  j\right)  \right)   &  =\sum_{j=1}%
^{N}a_{j}\sum_{m=0}^{X}\sum_{i=1}^{N}a_{i}\overline{\varphi_{m}\left(
x\left(  i\right)  \right)  }\varphi_{m}\left(  x\left(  j\right)  \right) \\
&  =\sum_{m=0}^{X}\left\vert \sum_{j=1}^{N}a_{j}\varphi_{m}\left(  x\left(
j\right)  \right)  \right\vert ^{2}\geq CX\sum_{j=1}^{N}a_{j}^{2}.
\end{align*}
Hence%
\[
1\geq CX\sum_{j=1}^{N}a_{j}^{2}.
\]
Applying Cauchy-Schwarz inequality to $1=\sum_{i=1}^{N}a_{i}$ we easily obtain
\[
\sum_{j=1}^{N}a_{j}^{2}\geqslant1/N
\]
and therefore $N\geq CX$.
\end{proof}

\end{document}